\documentclass[11pt]{amsart}
\usepackage{amssymb}
\usepackage{latexsym} 
\usepackage{srcltx} 
\usepackage{esint}
\usepackage{color}

\def\nc{\newcommand}

\def\om{\omega}

 \def\Om{\Omega}

\nc\pa{\partial}

\def\longequals{\mathbin{=\kern-2pt=}}

\DeclareMathOperator*{\esssup}{ess\,sup}

\nc\CC{\mathbb{C}}
\nc\RR{\mathbb{R}}
\nc\QQ{\mathbb{Q}}
\nc\ZZ{\mathbb{Z}}
\nc\NN{\mathbb{N}}

\nc\m[1]{\left| #1\right|}
\nc\norm[1]{\left\| #1\right\|}
\def\longequals{\mathbin{=\kern-2pt=}}

\newtheorem{theorem}{Theorem}[section]
\newtheorem{lemma}[theorem]{Lemma}
\newtheorem{corollary}[theorem]{Corollary}
\newtheorem{proposition}[theorem]{Proposition}
\newtheorem{definition}[theorem]{Definition}
\numberwithin{equation}{section}

\begin{document}

\title[Navier-Stokes equations in Lorentz spaces]
{The Navier-Stokes  equations in nonendpoint borderline Lorentz spaces}

\author[Nguyen Cong Phuc]
{Nguyen Cong Phuc}
\address{Department of Mathematics,
Louisiana State University,
303 Lockett Hall, Baton Rouge, LA 70803, USA.}
\email{pcnguyen@math.lsu.edu}


\begin{abstract} It is shown both locally and globally that $L_t^{\infty}(L_x^{3,q})$ solutions to the three-dimensional Navier-Stokes equations are regular provided $q\not=\infty$. Here $L_x^{3,q}$, $0<q\leq\infty$, is an increasing scale of  Lorentz spaces  containing $L^3_x$. Thus the result provides an  improvement of a result by Escauriaza, Seregin and {\v S}ver\'ak ((Russian) Uspekhi Mat. Nauk {\bf 58} (2003),  3--44; translation in Russian Math. Surveys {\bf 58} (2003),  211--250), which treated the case $q=3$.  A new local energy bound and a new $\epsilon$-regularity criterion are combined with the backward uniqueness theory of parabolic equations to obtain the  result. A weak-strong uniqueness of  Leray-Hopf weak solutions in  $L_t^{\infty}(L_x^{3,q})$, $q\not=\infty$, is also obtained as a consequence.
\end{abstract}

\maketitle

\section{Introduction}

This paper addresses certain regularity and uniqueness criteria  for the three-dimensional Navier-Stokes equations 
\begin{equation}\label{NSEqu}
\partial_t u -\Delta u + {\rm div}\, u\otimes u  +\nabla p=0, \quad {\rm div}\, u=0,
\end{equation}
where $u=u(x,t)=(u_1(x,t), u_2(x,t), u_3(x,t))\in \RR^3$ and $p=p(x,t)\in \RR$, with $x\in \RR^3$ and $t\geq 0$. The initial condition associated to \eqref{NSEqu} is given by
\begin{equation}\label{IniC}
u(x,0)=a(x), \quad x\in \RR^3.
\end{equation}

Equations \eqref{NSEqu}--\eqref{IniC} describes the motion an incompressible fluid in three spatial dimensions
with unit viscosity and zero external force. Here $u$ and $p$ are referred to as  the fluid velocity and   pressure, respectively. 

From the classical works of Leray \cite{Ler} and Hopf \cite{Hop}, it is known that for any divergence-free vector field $a\in L^2(\RR^3)$
there exists at least one weak solution to the Cauchy problem
\eqref{NSEqu}--\eqref{IniC} in $\RR^3\times (0,\infty)$. Such a solution is now called {\it Leray-Hopf weak solution} whose precise definition will be given next. Let $\dot{C}_0^\infty$ denote the space of all divergence-free infinitely
differentiable vector fields with compact support in $\RR^3$. Let $\dot{J}$  be the closure of  $\dot{C}_0^\infty$
in $L^2(\RR^3)$, and $\dot{J}^{1,2}$ be the closure of the same set with respect to the Dirichlet integral.

\begin{definition} A Leray-Hopf weak solution of the Cauchy problem \eqref{NSEqu}--\eqref{IniC} in $Q_\infty:=\RR^3\times (0, \infty)$ is a vector field 
$u: Q_\infty \rightarrow \RR^3$ such that:

{\rm (i)} $u\in L^\infty(0,\infty; \dot{J})\cap L^2(0,\infty; \dot{J}^{1,2})$;  

{\rm (ii)} The function $t\rightarrow \int_{\RR^3} u(x,t) w(x) dx$ is continuous on $[0,\infty)$ for any $w\in L^2(\RR^3)$;

{\rm (iii)} For any  $w\in \dot{C}_0^\infty(Q_\infty)$ there holds 
$$\int_{Q_\infty} (-u \cdot \partial_t w -u\otimes u:\nabla w +\nabla u:\nabla w) dx dt=0;$$

{\rm (iv)} The energy inequality 
$$ \int_{\RR^3} |u(x,t)|^2 dx + 2 \int_{0}^{t}\int_{\RR^3} |\nabla u|^2 dxds \leq  \int_{\RR^3} |a(x)|^2 dx$$
holds for all $t\in [0, \infty)$, and 
$$\norm{u(\cdot,t)-a(\cdot)}_{L^2(\RR^3)}\rightarrow 0 \quad {\rm as~} t\rightarrow 0^{+}.$$
\end{definition}


As of now the problems of uniqueness and regularity of Leray–
Hopf weak solutions are still open. Only some partial results are known. The partial uniqueness result of Prodi \cite{Pro} and Serrin \cite{Ser2}, and the partial smoothness result of  Ladyzhenskaya  \cite{Lad} can be summarized in the following theorem.

\begin{theorem} \label{PSL} Suppose that $a\in \dot{J}$ and $u$, $u_1$ are two Leray-Hopf weak solutions to the Cauchy problem \eqref{NSEqu}--\eqref{IniC}.
If $u\in L^s(0,T; L^p(\RR^3))$ for some $T>0$, where 
$$\frac{3}{p} + \frac{2}{s}=1, \qquad p\in (3, \infty],$$
then $u=u_1$ in $Q_T:=\RR^3\times (0, T)$ and, moreover, $u$ is smooth on $\RR^3\times (0, T]$. 
\end{theorem}
Here recall that the condition  $u\in L^s(0,T; L^p(\RR^3))$ means that 
$$\norm{u}_{L^s(0,T; L^p(\RR^3))}:=\Big(\int_{0}^{T}\norm{u(\cdot,t)}^s_{L^p(\RR^n)} dt\Big)^{\frac{1}{s}}<+\infty \qquad {\rm if~} s\in[1,\infty),$$
and
$$\norm{u}_{L^s(0,T; L^p(\RR^3))}:=\esssup_{t\in(0,T)}\norm{u(\cdot,t)}_{L^p(\RR^n)}<+\infty \qquad {\rm if~} s=\infty.$$
It is obvious that, when $s=p$, $L^p(0,T; L^p(\RR^3))= L^p(Q_T)$. In general, 
if X is a  Banach space with norm $\norm{\cdot}_{X}$, then
$L^s(a,b; X)$, $a<b$, means the usual Banach space of measurable X-valued  functions $f(t)$ on $(a,b)$ such that the norm
\begin{equation}\label{norms}
\norm{u}_{L^s(a,b; X)}:=\Big(\int_{0}^{T}\norm{f(t)}^s_{X} dt\Big)^{\frac{1}{s}}<+\infty
\end{equation}
for $s\in [1, \infty)$, and with the usual modification of \eqref{norms} in the case $s=\infty$.

The endpoint case $p=3$ and $s=\infty$, which is not covered by Theorem \ref{PSL},  was considered harder and has been settled by Escauriaza-Seregin-{\v S}ver\'ak
in the interesting work \cite{ESS1}:

\begin{theorem}\label{global-ESS} Let $a\in \dot{J}\cap L^3(\RR^3)$.
Suppose that $u$ is a  Leray-Hopf weak solution of the Cauchy problem
\eqref{NSEqu}--\eqref{IniC}, and u satisfies the additional condition 
\begin{equation}\label{PSLcrit}
u\in L^\infty(0,T; L^3(\RR^3))
\end{equation}
for some $T>0$. Then $u\in L^5(Q_T)$ and hence it is unique and smooth on $\RR^3\times(0,T]$.
\end{theorem}

We remark that the condition $a\in  L^3(\RR^3)$ in the above theorem is superfluous as it can be deduce from  condition \eqref{PSLcrit}.
A basic consequence of Theorem \ref{global-ESS} is that if a Leray-Hopf weak solution $u$ develops a singularity at a first finite time $t_0>0$
then there necessarily  holds
\begin{equation}\label{L3blowup}
\limsup_{t\uparrow t_0} \norm{u(\cdot,t)}_{L^3(\RR^3)}=+\infty.
\end{equation}
An improvement of this necessary condition of potential blow up can be found in the recent work \cite{Sere}.  See also the papers \cite{GKP, KeKo}  
for another approach to regularity using certain profile decompositions. 

It should be noticed that the uniqueness of  $u$ under condition \eqref{PSLcrit} had been known earlier (see  \cite{KS}).
Moreover, local versions of the corresponding partial regularity results are also available (see \cite{Ser1}, \cite{Stru}, and \cite{ESS1}).
In particular, the local regularity result of \cite{ESS1} reads as follows.

\begin{theorem} \label{localregularity-ESS} Suppose that the pair of functions $(u,p)$ satisfies 
the Navier-Stokes equations \eqref{NSEqu} in $Q_1(0,0)=B_1(0)\times(-1,0)$ in the sense of distributions and has the following properties:
\begin{equation}\label{u-reg}
u\in L^\infty(-1,0; L^2(B_1)) \cap L^2(-1,0; W^{1, 2}(B_1))
\end{equation}
and 
\begin{equation*}
p\in L^{3/2}(-1, 0; L^{3/2}(B_{1})).
\end{equation*}
Suppose further that 
\begin{equation*}
u\in L^\infty(-1,0; L^{3}(B_1)).
\end{equation*}
Then the velocity function $u$ is H\"older continuous on $\overline{Q}_{1/2}(0,0)$.
\end{theorem}

The main goal of this paper is to improve Theorems \ref{global-ESS} and \ref{localregularity-ESS} by means of Lorentz spaces.
Given a measurable set $\Om\subset\RR^3$, recall that the Lorentz space
$L^{p,q}(\Om)$, with  $p\in(0, \infty), q\in(0,\infty]$, is the set of
measurable functions $g$ on $\Omega$ such that the quasinorm $\|g\|_{L^{p, q}(\Omega)}$ is finite. Here we define
\begin{equation*}
\|g\|_{L^{p, q}(\Omega)}:=\left\{ \begin{array}{lcr}
 \displaystyle \left(p \int_{0}^{\infty}\alpha^{q} |\{x\in\Om: |g(x)|>\alpha\}|^{\frac{q}{p}} \frac{d\alpha}{\alpha}
\right)^{\frac{1}{q}}   \text{~if~}  q<\infty, \\
\, \   \displaystyle \sup_{\alpha >0} \,\alpha \, |\{x\in \Om: |g(x)|>\alpha\}|^{\frac{1}{p} }  \text{~if~}  q=\infty.
\end{array}\right.
\end{equation*}

The space $L^{p,\infty}(\Om)$ is often referred to as the Marcinkiewicz or weak $L^p$ space.  It is known that $L^{p,p}(\Om)=L^p(\Om)$ and $L^{p,q_1}(\Om) \subset L^{p,q_2}(\Om)$ whenever $q_1\leq q_2$.
On the other hand, if $|\Om|$ is finite then $L^{p,q}(\Om)\subset L^r(\Om)$ for any $0<q\leq\infty$ and $0<r<p$. Moreover,
$$\norm{g}_{L^{r}(\Om)}\leq |\Om|^{\frac{1}{r}-\frac{1}{p}} \norm{g}_{L^{p,q}(\Om)}.$$

Lorentz spaces can be used to capture logarithmic singularities. For example, in $\RR^3$, for any $\beta>0$ we have 
\begin{equation}\label{loglorentz}
|x|^{-1} | \log (|x|/2)|^{-\beta}  \in L^{3,q}(B_1(0)) \quad {\rm if~and~only~if~} q>\frac{1}{\beta}.
\end{equation}
Note that the inequality in \eqref{loglorentz} is strict.
Of course, in the case $\beta=0$, the function $|x|^{-1}$ belongs to the Marcinkiewicz space $L^{3,\infty}(\RR^3)$.

To the best of our knowledge, a criterion of local regularity for the Navier-Stokes equations in $L^\infty(-1,0; L^{3,\infty}(B_1))$
is still unknown. See \cite{KK, LT} for some partial results, which require a smallness condition. See also \cite{Soh, Tak} for some nonendpoint
related results. The first result of this paper provides  instead a regularity condition in terms of the borderline space 
$L^\infty(-1,0; L^{3,q}(B_1))$ for {\it any} $q\in (0,\infty)$, and thus excluding only the endpoint case $q=\infty$.

\begin{theorem} \label{localregularity} Suppose that the pair of functions $(u,p)$ satisfies 
the Navier-Stokes equations \eqref{NSEqu} in $Q_1(0,0)=B_1(0)\times(-1,0)$ in the sense of distributions such that \eqref{u-reg} holds 
and 
\begin{equation}\label{p-assum}
p\in L^{2}(-1, 0; L^{1}(B_{1})).
\end{equation}
Suppose further that 
\begin{equation}\label{serrinlorentz}
u\in L^\infty(-1,0; L^{3,q}(B_1)) 
\end{equation}
for some  $q\in (3,\infty)$. Then the velocity function $u$ is H\"older continuous on $\overline{Q}_{1/2}(0,0)$.
\end{theorem}

It is worth mentioning that even the regularity at  $(0,0)$ is still unknown for solutions $u$ satisfying the pointwise bound
$$|u(x,t)|\leq C\, |x|^{-1}  $$ 
for a.e. $(x,t)\in Q_1(0,0)$. A regularity result under this condition is known only for axially symmetric  solutions (see \cite{SS2} and also \cite{CSTY1, CSTY2}).  On the other hand, in view of \eqref{loglorentz}, Theorem \ref{localregularity} yields the regularity of $u$ under a logarithmic `bump' condition
$$|u(x,t)|\leq C\, |x|^{-1} | \log (|x|/2)|^{-\beta} $$ 
for any $\beta>0$.

In fact, it is possible to obtain regularity under a weaker pointwise bound condition on the solution. In this case 
 Theorem \ref{localregularity} is no longer applicable.
 
\begin{theorem}\label{nearMarc}  Suppose that the pair of functions $(u,p)$ satisfies 
the Navier-Stokes equations \eqref{NSEqu} in $Q_1(0,0)$ in the sense of distributions such that \eqref{u-reg} holds, and 
$$p\in L^{3/2}(-1, 0; L^{1}(B_{1})).$$ 
Suppose further that for a.e. $(x,t)\, \in Q_1(0,0)$, there holds 
\begin{equation}\label{LinftyX}
|u(x,t)|\leq f(t)|x|^{-1} g(x) 
\end{equation}
for nonnegative functions $f\in L^\infty((-1,0))$ and  $g\in L^{\infty}(B_1(0))$ such that $\lim_{x\rightarrow 0} g(x)=0$. Then  $u$ is H\"older continuous on $\overline{Q}_{1/2}(0,0)$.
\end{theorem}

On the other hand, Theorem \ref{localregularity} can be used to deduce the following uniqueness and global regularity results, which give an 
improvement of Theorem \ref{global-ESS}.

\begin{theorem}\label{global-ESS-lorentz} Let $a\in \dot{J}$.
Suppose that $u$ is a  Leray-Hopf weak solution of the Cauchy problem
\eqref{NSEqu}--\eqref{IniC}, and it satisfies the additional condition 
\begin{equation}\label{PSLcrit-lorentz}
u\in L^\infty(0,T; L^{3,q}(\RR^3)) 
\end{equation}
for some $q\in (3,\infty)$ and $T>0$.  Then $u$ is smooth  on $\RR^3\times(0,T]$. Moreover, if in addition $a\in L^3(\RR^3)$ then $u\in L^5(Q_T)$ and hence it is unique in $Q_T$ (in the sense of weak-strong uniqueness as in Theorem \ref{PSL}).
\end{theorem}

Theorem \ref{global-ESS-lorentz} implies that the necessary condition of potential blow up \eqref{L3blowup} can now be improved  by replacing the $L^3$ norm
with any smaller $L^{3, q}$ quasi-norm provided $q\not=\infty$. We should mention that this kind of potential blow up criterion has recently been 
extended in \cite{GKP2} to the norms of Besov spaces $\dot{B}_{q, p}^{-1+3/p}(\RR^3)$, $3<p,q<\infty$, using   profile decompositions in  the framework of ``strong" solutions. See also \cite{CP} for an earlier related result. In such a setting of strong solutions,  the blow up criterion of \cite{GKP2}
 is more general than ours since  for $q>3$,
$$L^{3,q}(\RR^3)\subset \dot{B}_{q, q}^{-1+3/q}(\RR^3).$$

However, our blow up criterion  here is obtained for  Leray-Hopf weak solutions.
Moreover, using instead the local regularity criterion, Theorem \ref{localregularity}, we see that if $(x_0,t_0)$ is a singular 
point of a Leray-Hopf weak solution $u$ then there holds
$$\limsup_{t\uparrow t_0} \norm{u(\cdot,t)}_{L^{3,q}(B_{\delta}(x_0))}=+\infty$$
for any $\delta>0$ and $q\not=\infty$. Note that for a Leray-Hopf weak solution $u$, the associated pressure $p$ can be chosen so that 
$p\in L^2(0,\infty; L^{3/2}(\RR^3))$ since $|u|^2$ belongs to the same space (see, e.g., \cite{SS}).

Our approach to Theorems \ref{localregularity} and \ref{global-ESS-lorentz} is  influenced by the above mentioned work of Escauriaza-Seregin-{\v S}ver\'ak 
\cite{ESS1}, which reduces the regularity matter to the  backward uniqueness problem for parabolic equations
with variable lower-order terms. A key  ingredient, which makes our results stronger than that
of \cite{ESS1}, is   a new  $\epsilon$-regularity criterion for {\it suitable weak solutions}  to the Navier-Stokes equations (see Proposition \ref{epsilon1}).  
See Definition \ref{SWS} below for the notion of suitable weak solutions.
In turn, this kind of $\epsilon$-regularity criterion is a consequence of a new bound for some scaling invariant energy quantities (see Corollary \ref{ABcontrol2}). Moreover, this new energy bound is also essential in a blow-up procedure needed in the proof of Theorem \ref{localregularity}.
It provides a certain compactness result and thus  yields a  non-trivial `ancient solution' (see Proposition \ref{limit-infty}), another important 
ingredient in the proof of Theorem \ref{localregularity}.

On the other hand, the proof of Theorem \ref{nearMarc} is simple. It requires only an $\epsilon$-regularity criterion of Seregin and {\v S}ver\'ak in \cite[Lemma 3.3]{SS}.

\section{ Preliminaries and local energy estimates}

Throughout the paper we use the following notations for balls and parabolic cylinders:
$$B_r(x)=\{y\in \RR^3: |x-y|<r\}, \quad x\in\RR^3, \, r>0,$$
and
$$Q_r(z)=B_r(x)\times (t-r^2, t) \quad {\rm with~} z=(x,t).$$

The following scaling invariant  quantities will be employed:
\begin{eqnarray*}
A(z_0, r)=A(u, z_0, r)&=& \sup_{t_0-r^2\leq t\leq t_0} r^{-1}\int_{B_r(x_0)} |u(x,t)|^2 dx,\\
B(z_0, r)=B(u, z_0, r)&=& r^{-1}\int_{Q_r(x_0)} |\nabla u(x,t)|^2 dxdt,\\
C(z_0, r)=C(u, z_0, r)&=& r^{-3}\int_{t_0-r^2}^{t_0} \norm{u}_{L^{\frac{12}{5}}(B_r(x_0))}^4 dt,\\
C_1(z_0, r)=C_1(u, z_0, r)&=& r^{-2}\int_{t_0-r^2}^{t_0} \norm{u}_{L^{3}(B_r(x_0))}^3 dt,\\
D(z_0, r)=D(p,z_0, r)&=& r^{-3}\int_{t_0-r^2}^{t_0} \norm{p}_{L^{\frac{6}{5}}(B_r(x_0))}^2 dt,\\
D_1(z_0, r)=D_1(p, z_0, r)&=& r^{-2}\int_{t_0-r^2}^{t_0} \norm{p}_{L^{\frac{3}{2}}(B_r(x_0))}^{3/2} dt.\\
\end{eqnarray*}

To analyze local properties of solutions, it is often useful to use the notion of suitable weak solutions.  Such a notion of weak solutions was introduced in Caffarelli-Kohn-Nirenberg \cite{CKN} following the work of Scheffer \cite{Sche1}--\cite{Sche4}. Here we use the version introduced by Lin in \cite{Lin}.

\begin{definition}\label{SWS} Let $\om$ be an open set in $\RR^3$ and let $-\infty<a< b< \infty$. We say that a pair $(u, p)$ is a suitable weak solution 
to the Navier-Stokes equations in $Q=\om\times (a, b)$ if the following conditions hold:

{\rm (i)}   $u\in L^\infty(a, b; L^2(\om))\cap L^2(a, b; W^{1,\, 2}(\om)) {\rm ~and~} p\in L^{3/2}(\om\times (a, b));$

{\rm (ii)} $(u,p)$  satisfies the Navier-Stokes equations in the sense of distributions. That is, 

$$\int_{a}^{b}\int_{\om} \left\{-u \, \psi_{t} + \nabla u : \nabla \psi - (u\otimes u):\nabla \psi -  p \, {\rm div}\, \psi \right\} dxdt =0$$ 
for all vector fields $\psi\in C_0^{\infty}(\om\times(a,b); \RR^3)$, and
$$\int_{\om\times \{t\}} u(x,t)\cdot \nabla \phi(x) \, dx=0$$
for a.e. $t\in (a, b)$ and  all real valued functions $\phi\in C_0^{\infty}(\om)$;

{\rm (iii)} $(u,p)$  satisfies the local generalized energy inequality
\begin{eqnarray*}
\lefteqn{ \int_{\om}|u(x, t)|^2 \phi(x,t) dx + 2\int_{a}^t\int_{\om} |\nabla u|^2 \phi(x, s) dxds} \\
&\leq& \int_{a}^t\int_{\om} |u|^2 (\phi_t +\Delta \phi) dx ds + \int_{a}^t\int_{\om}(|u|^2 + 2p)u\cdot \nabla \phi dx ds.
\end{eqnarray*}
for a.e. $t\in (a, b)$ and any nonnegative function $\phi \in C_0^{\infty}(\RR^3\times\RR)$ vanishing in a neighborhood of the parabolic boundary 
$\partial'Q=\om\times\{t=a\} \cup \partial\om\times [a, b]$.
\end{definition}

A proof of the following lemma can be found in \cite[Lemma 6.1]{Giu}.

\begin{lemma}\label{Giusti-lem}
Let $I(s)$ be a bounded nonnegative function in the interval $[R_1, R_2]$. Assume that for every $s, \rho\in [R_1, R_2]$ and  $s<\rho$ we have 
$$I(s)\leq [A(\rho-s)^{-\alpha} +B(\rho-s)^{-\beta} +C] +\theta I(\rho)$$
with  $A, B, C\geq 0$, $\alpha>\beta>0$ and $\theta\in [0,1)$. Then there holds
$$I(R_1)\leq c(\alpha, \theta) [A(R_2-R_1)^{-\alpha} +B(R_2-R_1)^{-\beta} +C].$$

\end{lemma}

In the next lemma $L^{-1, \, 2}(B_r(x_0))$ stands for the dual  of the Sobolev space $W^{1, \, 2}_0(B_r(x_0))$. The latter is defined as the completion of $C_0^{\infty}(B_r(x_0))$ under the Dirichlet norm 
$$\norm{\varphi}_{W^{1, \, 2}_0(B_r(x_0))}=\Big(\int_{B_r(x_0)} |\nabla \varphi|^2 dx \Big)^{1/2}.$$

\begin{lemma}\label{ABcontrol0}  Suppose that $(u,p)$ is a suitable weak solution 
to the Navier-Stokes equations in $Q=\om\times (a, b)$. Let $z_0=(x_0, t_0)$ and $r>0$ be such that $Q_r(z_0)\subset Q$. Then there holds 
\begin{eqnarray*}
A(z_0, r/2) + B(z_0, r/2)&\leq& C \left[r^{-3} \int_{t_0-r^2}^{t_0} \norm{|u|^2}^2_{L^{-1, \, 2}(B_r(x_0))} dt\right]^{1/2}\nonumber\\
&& + \, C r^{-3} \int_{t_0-r^2}^{t_0} \norm{|u|^2 +2p}^2_{L^{-1, \, 2}(B_r(x_0))} dt.\nonumber
\end{eqnarray*}
\end{lemma}
\begin{proof}
For $z_0=(x_0, t_0)$ and $r>0$ such that $Q_r(z_0)\subset Q$, we consider the cylinders $$Q_s(z_0)=B_s(x_0)\times (t_0-s^2, t_0)\subset Q_\rho(z_0)=B_\rho(x_0)\times (t_0-\rho^2, t_0),$$
where $r/2\leq s<\rho\leq r$.

Let $\phi(x,t)=\eta_1(x)\eta_2(t)$ where $\eta_1\in  C_0^{\infty}(B_\rho(x_0))$, $0\leq \eta_1\leq 1$ in $\RR^n$, $\eta_1\equiv 1$ on $B_s(x_0)$, and $$|\nabla^{\alpha} \eta_1|\leq \frac{c}{(\rho-s)^{|\alpha|}}$$ for all multi-indices $\alpha$ with $|\alpha|\leq 3$. The function $\eta_2(t)$ is chosen so that 
$\eta_2\in C_0^{\infty}(t_0-\rho^2, t_0+\rho^2)$, $0\leq\eta_2\leq 1$ in $\RR$, $\eta_2(t)\equiv 1$ for $t\in [t_0-s^2, t_0+s^2]$, and $$|\eta'_2(t)|\leq \frac{c}{\rho^2-s^2}\leq \frac{c}{r(\rho-s)}.$$

Then
$$|\nabla \phi_t|\leq \frac{c}{r(\rho-s)^2}\leq \frac{c}{(\rho-s)^3},\quad |\nabla \Delta \phi|\leq  \frac{c}{(\rho-s)^3},$$
$$| \nabla^2 \phi|\leq  \frac{c}{(\rho-s)^2},\quad |\nabla  \phi|\leq  \frac{c}{\rho-s}.$$

We next define
$$I(s)=I_1(s) +I_2(s),$$
where
$$I_1(s)=\sup_{t_0-s^2\leq t\leq t_0}\int_{B_s(x_0)}|u(x, t)|^2 dx=s\, A(z_0,s)$$
and 
$$I_2(s)=\int_{t_0-s^2}^{t_0}\int_{B_s(x_0)}|\nabla u(x, t)|^2 dx dt=s\, B(z_0,s).$$

Using $\phi$ as a test function in the generalized energy inequality we find
\begin{eqnarray}\label{Isenergy}
\lefteqn{I(s)}\\
&\leq& \int_{t_0-\rho^2}^{t_0} \norm{|u|^2}_{L^{-1, \, 2}(B_\rho(x_0))}\norm{\nabla \phi_t +\nabla \Delta \phi}_{L^2(B_{\rho}(x_0))} dt \nonumber\\
&& + \int_{t_0-\rho^2}^{t_0} \Big\{\norm{|u|^2 + 2p}_{L^{-1, \, 2}(B_\rho(x_0))}\times\nonumber\\
&& \qquad \qquad\qquad\times\norm{\nabla u\cdot \nabla \phi + u\cdot \nabla^2\phi}_{L^2(B_{\rho}(x_0))}\Big\} dt\nonumber \\
&=:& J_1+ J_2.\nonumber
\end{eqnarray}

By the choice of test function we have 
\begin{eqnarray}\label{firstterm}
J_1&\leq& C \frac{\rho^{3/2}}{(\rho-s)^3} \int_{t_0-\rho^2}^{t_0} \norm{|u|^2}_{L^{-1, \, 2}(B_\rho(x_0))} dt\\
&\leq& C \frac{\rho^{5/2}}{(\rho-s)^3} \left[ \int_{t_0-\rho^2}^{t_0} \norm{|u|^2}^2_{L^{-1, \, 2}(B_\rho(x_0))} dt\right]^{1/2}.\nonumber
\end{eqnarray}

Also,
\begin{eqnarray*}
J_2 &\leq& C \int_{t_0-\rho^2}^{t_0} \Big\{ \norm{|u|^2+2p}_{L^{-1, \, 2}(B_\rho(x_0))} \times\\
 && \qquad \qquad\qquad \times \Big[\frac{\norm{\nabla u}_{L^2(B_\rho(x_0))}}{\rho-s} +
\frac{\norm{u}_{L^2(B_\rho(x_0))}}{(\rho-s)^2}\Big] \Big\}dt,
\end{eqnarray*}
and thus by H\"older's inequality we get  
\begin{eqnarray}\label{secondterm}
J_2 &\leq& \frac{C}{\rho-s} \left[\int_{t_0-\rho^2}^{t_0} \norm{|u|^2 +2p}^2_{L^{-1, \, 2}(B_\rho(x_0))} dt\right]^{1/2} I_{2}(\rho)^{1/2} \\
&&+ \frac{C \rho}{(\rho-s)^2} \left[\int_{t_0-\rho^2}^{t_0} \norm{|u|^2 +2p}^2_{L^{-1, \, 2}(B_\rho(x_0))} dt\right]^{1/2} I_{1}(\rho)^{1/2}.\nonumber
\end{eqnarray}

Combining inequalities \eqref{Isenergy}--\eqref{secondterm} and using $\rho\leq r$ we arrive at
\begin{eqnarray*}
I(s) &\leq&  \frac{C r^{5/2}}{(\rho-s)^3} \left[ \int_{t_0-\rho^2}^{t_0} \norm{|u|^2}^2_{L^{-1, \, 2}(B_\rho(x_0))} dt\right]^{1/2}+\\
&& +\frac{C}{\rho-s} \left[\int_{t_0-\rho^2}^{t_0} \norm{|u|^2 +2p}^2_{L^{-1, \, 2}(B_\rho(x_0))} dt\right]^{1/2} I_{2}(\rho)^{1/2} +\\
&&+ \frac{C r}{(\rho-s)^2} \left[\int_{t_0-\rho^2}^{t_0} \norm{|u|^2 +2p}^2_{L^{-1, \, 2}(B_\rho(x_0))} dt\right]^{1/2} I_{1}(\rho)^{1/2}.
\end{eqnarray*}

By Young's inequality this yields 
\begin{eqnarray*}
I(s) &\leq&  \frac{C r^{5/2}}{(\rho-s)^3} \left[ \int_{t_0-\rho^2}^{t_0} \norm{|u|^2}^2_{L^{-1, \, 2}(B_\rho(x_0))} dt\right]^{1/2}\\
&& +\left\{\frac{C}{(\rho-s)^2} +\frac{C r^2}{(\rho-s)^4} \right\}\int_{t_0-\rho^2}^{t_0} \norm{|u|^2+2p}^2_{L^{-1, \, 2}(B_\rho(x_0))} dt \\
&& + \frac{1}{2}I(\rho),
\end{eqnarray*}
which implies in particular that 
\begin{eqnarray*}
I(s) &\leq&  \frac{C r^{5/2}}{(\rho-s)^3} \left[ \int_{t_0-r^2}^{t_0} \norm{|u|^2}^2_{L^{-1, \, 2}(B_r(x_0))} dt\right]^{1/2}\\
&& +\frac{C r^2}{(\rho-s)^4} \int_{t_0-r^2}^{t_0} \norm{|u|^2+2p}^2_{L^{-1, \, 2}(B_r(x_0))} dt + \frac{1}{2}I(\rho).
\end{eqnarray*}

Since this holds for all $r/2\leq s<\rho\leq r$ by Lemma \ref{Giusti-lem} we find 
\begin{eqnarray*}
I(r/2) &\leq& C r^{-1/2} \left[ \int_{t_0-r^2}^{t_0} \norm{|u|^2}^2_{L^{-1, \, 2}(B_r(x_0))} dt\right]^{1/2}\\
&& + \, C r^{-2} \int_{t_0-r^2}^{t_0} \norm{|u|^2 +2p}^2_{L^{-1, \, 2}(B_r(x_0))} dt.
\end{eqnarray*}

Thus 
\begin{eqnarray*}
A(z_0, r/2) + B(z_0, r/2)&\leq& C \left[r^{-3} \int_{t_0-r^2}^{t_0} \norm{|u|^2}^2_{L^{-1, \, 2}(B_r(x_0))} dt\right]^{1/2}\nonumber\\
&& + \, C r^{-3} \int_{t_0-r^2}^{t_0} \norm{|u|^2 +2p}^2_{L^{-1, \, 2}(B_r(x_0))} dt
\end{eqnarray*}
as desired.
\end{proof}

Note that for $f\in L^{6/5}(B_r(x_0))$ and for $\varphi\in C_0^{\infty}(B_r(x_0))$ we have 
\begin{eqnarray*}
\left |\int_{B_r(x_0)} \varphi(x) f(x)dx \right | &\leq& C \int_{B_r(x_0)} \left[\int_{B_r(x_0)}\frac{|\nabla \varphi(y)|}{|x-y|^{2}} dy\right]
|f(x)| dx\\
&=& C  \int_{B_r(x_0)}|\nabla \varphi(y)| \left[\int_{B_r(x_0)}\frac{ |f(x)|dx}{|x-y|^{2}}\right]
dy\\
&\leq& C \norm{\nabla \varphi}_{L^2(B_r(x_0))} \norm{ {\rm\bf I}_1(\chi_{B_r(x_0)}|f|)}_{L^2(B_r(x_0))}.
\end{eqnarray*}
Here ${\rm\bf I}_1$ is the first order Riesz's potential defined by 
$${\rm\bf I}_1(\mu)(x)= c \int_{\RR^3}\frac{d\mu(y)}{|x-y|^{2}}, \qquad x\in \RR^3, $$ 
for a nonnegative locally finite measure $\mu$ in $\RR^3$.
Thus we find 
\begin{equation}\label{dualnorm}
\norm{f}_{L^{-1,\, 2}(B_r(x_0))} \leq C \norm{ {\rm\bf I}_1(\chi_{B_r(x_0)}|f|)}_{L^2(B_r(x_0))}\leq C\norm{f}_{L^{\frac{6}{5}}(B_r(x_0))}
\end{equation}
by the embedding property of Riesz's potentials. 

Using  \eqref{dualnorm} we obtain the following important consequence  of Lemma \ref{ABcontrol0}.

\begin{corollary} \label{ABcontrol2} Suppose that $(u,p)$ is a suitable weak solution 
to the Navier-Stokes equations in $Q=\om\times (a, b)$. Let $z_0=(x_0, t_0)$ and $r>0$ be such that $Q_r(z_0)\subset Q$. Then there holds 
\begin{equation*}
A(z_0, r/2) + B(z_0, r/2)\leq C[ C(z_0, r)^{1/2} +  C(z_0, r) +  D(z_0, r)].
\end{equation*}
\end{corollary}

\section{$\epsilon$-regularity criteria}
As demonstrated in \cite{ESS1}, the proof of Theorem \ref{localregularity-ESS} above relies heavily on the following $\epsilon$-regularity criterion for suitable weak solutions to the Navier-Stokes equations (see \cite[Lemma 2.2]{ESS1}, see also \cite{CKN, LS, NRS}).
 
\begin{proposition}\label{C1D1ep0}
There exist positive constants $\epsilon_0$ and $C_k$, $k=0, 1, 2,  \dots$, such that the following holds. Suppose that the pair
$(u,p)$ is a suitable solution to the Navier-Stokes equations in $Q_1(z_0)$ and satisfies the smallness condition
\begin{equation*}
C_1(u,z_0,1) + D_1(p,z_0,1) \leq \epsilon_0.
\end{equation*}
Then $\nabla^{k} u$ is H\"older continuous on $\overline{Q}_{1/2}(z_0)$ for any  integer $k\geq 0$, and 
\begin{equation*}
\max_{z\in \overline{Q}_{1/2}(z_0)} |\nabla^{k} u(z)| \leq C_k.
\end{equation*}
\end{proposition}

To prove Theorem \ref{localregularity} we use instead a new different version of $\epsilon$-regularity criterion.

\begin{proposition}\label{epsilon1}
There exist positive constants $\epsilon_1$ and $C_k$, $k=0, 1, 2,  \dots$, such that the following holds. Suppose that the pair
$(u,p)$ is a suitable solution to the Navier-Stokes equations in $Q_8(z_0)$ and satisfies the smallness condition
\begin{equation}\label{smallnessCD}
C(u,z_0,8) + D(p,z_0,8) \leq \epsilon_1.
\end{equation}
Then $\nabla^{k} u$ is H\"older continuous on $\overline{Q}_{1/2}(z_0)$ for any  integer $k\geq 0$, and 
\begin{equation*}
\max_{z\in \overline{Q}_{1/2}(z_0)} |\nabla^{k} u(z)| \leq C_k.
\end{equation*}
\end{proposition}

The proof of Proposition \ref{epsilon1} will be given at the end of this section. It requires the following two preliminary results. The first one is by now a well-known lemma  that  can be found in  \cite[Lemma 2.1]{Lin}.

\begin{lemma}\label{boundC1}
Suppose that $(u,p)$ is a suitable weak solution 
to the Navier-Stokes equations in $Q=\om\times (a, b)$.
Let $z_0=(x_0, t_0)$ and let  $\rho>0$ be such that $Q_{\rho}(z_0)\subset Q=\om\times (a, b)$. For any $r\in (0, \rho]$ we have
$$C_1(z_0,r)\leq C \Big(\frac{\rho}{r} \Big)^3 A(z_0,\rho)^{3/4} B(z_0,\rho)^{3/4} + C \Big(\frac{r}{\rho} \Big)^3 A(z_0,\rho)^{3/2}.$$
\end{lemma}

In what follows we shall use the following notation to denote the spatial average of a function $f$ over a ball $B_r(x_0)$:
$$[f]_{x_0, r}:=\frac{1}{|B_r(x_0)|}\int_{B_r(x_0)} f(x) \, dx.$$

\begin{lemma}\label{Dbar}
Suppose that $(u,p)$ is a suitable weak solution 
to the Navier-Stokes equations in $Q=\om\times (a, b)$.
Let $z_0=(x_0, t_0)$ and let  $\rho>0$ be such that $Q_{\rho}(z_0)\subset Q=\om\times (a, b)$. For any $r\in (0, \rho/4]$ we have
\begin{equation*}
D_1(z_0, r)\leq C \Big(\frac{\rho}{r}\Big)^{3/2} A(z_0,\rho)^{3/4} B(z_0,\rho)^{3/4} + C \Big(\frac{r}{\rho}\Big)^{3/2} D(z_0,\rho)^{3/4}.
\end{equation*}
\end{lemma}

\begin{proof} Let $h_{x_0, \rho}=h_{x_0, \rho}(\cdot, t)$ be a function on $B_\rho(x_0)$ for a.e. $t$ such that 
$$h_{x_0, \rho}=p-\tilde{p}_{x_0, \rho}\quad {\rm in~}  B_\rho(x_0),$$ 
where $\tilde{p}_{x_0, \rho}$ is defined by 
$$\tilde{p}_{x_0, \rho}=R_iR_j[(u_i-[u_i]_{x_0,\rho})(u_j-[u_j]_{x_0,\rho})\chi_{B_\rho(x_0)}].$$
Here $R_i=D_i(-\Delta)^{-\frac{1}{2}}$, $i=1,2,3$, is the $i$-th Riesz transform. Note that for any $\varphi\in C_0^{\infty}(B_\rho(x_0))$, we have
\begin{eqnarray*}
-\int_{B_\rho(x_0)} \tilde{p}_{x_0, \rho}\Delta \varphi dx&=&\int_{B_{\rho}(x_0)}(u_i-[u_i]_{x_0,\rho})(u_j-[u_j]_{x_0,\rho}) D_{ij}\varphi\, dx\\
&=&\int_{B}u_i u_j D_{ij}\varphi\, dx,
\end{eqnarray*}
which follows from  the properties $-R_i R_j(\Delta \varphi)=D_{ij}\varphi$ and ${\rm div}\, u=0$.
Thus, as $p$ also solves
$$-\Delta p={\rm div}\, {\rm div} (u\otimes u)$$
in the  distributional sense, we see that  $h_{x_0, \rho}$ is harmonic in  $B_{\rho}(x_0)$ for a.e. $t$.

With this decomposition of the pressure $p$, we have 
\begin{eqnarray*}
\int_{B_r(x_0)} |p(x,t)|^{3/2} dx &=& C \int_{B_r(x_0)} |\tilde{p}_{x_0,\rho}+h_{x_0,\rho}|^{3/2} dx\\
&\leq& C \int_{B_\rho(x_0)} |\tilde{p}_{x_0, \rho}|^{3/2} dx + C\int_{B_r(x_0)} |h_{x_0, \rho}|^{3/2} dx.
\end{eqnarray*}

Next, as $h_{x_0,\rho}$ is harmonic in $B_{\rho}(x_0)$, for $r\in (0, \rho/4]$ there holds
\begin{eqnarray*}
\Big(\fint_{B_r(x_0)} |h_{x_0, \rho}|^{3/2} dx\Big)^{2/3}&\leq&  \Big(\fint_{B_{r}(x_0)} |h_{x_0, \rho}|^{2} dx\Big)^{1/2}\\
&\leq& C\,  \Big(\fint_{B_{\rho/4}(x_0)} |h_{x_0, \rho}|^{2} dx\Big)^{1/2}\\
&\leq& C \Big(\fint_{B_{\rho/2}(x_0)} |h_{x_0, \rho}|^{6/5} dx\Big)^{5/6}.
\end{eqnarray*}

This gives
\begin{eqnarray*}
\int_{B_r(x_0)} |p(x,t)|^{3/2} dx &\leq& C \int_{B_\rho(x_0)} |\tilde{p}_{x_0, \rho}|^{3/2} dx\\
&& +\,  C \frac{r^3}{\rho^{15/4}}\Big(\int_{B_{\rho/2}(x_0)} |h_{x_0, \rho}|^{6/5} dx\Big)^{5/4}.
\end{eqnarray*}

Thus using $h_{x_0,\rho}=p-\tilde{p}_{x_0,\rho}$ again we find
\begin{eqnarray*}
\int_{B_r(x_0)} |p(x,t)|^{3/2} dx &\leq& C \int_{B_\rho(x_0)} |\tilde{p}_{x_0, \rho}|^{3/2} dx\\
&& +\,  C \frac{r^3}{\rho^{15/4}}\Big(\int_{B_{\rho}(x_0)} |\tilde{p}_{x_0, \rho}|^{6/5} dx\Big)^{5/4}\\
&& +\,  C \frac{r^3}{\rho^{15/4}}\Big(\int_{B_{\rho}(x_0)} |p|^{6/5} dx\Big)^{5/4}.
\end{eqnarray*}

By H\"older's inequality this yields
\begin{eqnarray}\label{pptil}
\int_{B_r(x_0)} |p(x,t)|^{3/2} dx &\leq& C\Big[1+\Big(\frac{r}{\rho}\Big)^3\Big] \int_{B_\rho(x_0)} |\tilde{p}_{x_0, \rho}|^{3/2} dx\\
&& +\,  C \frac{r^3}{\rho^{15/4}}\Big(\int_{B_{\rho}(x_0)} |p|^{6/5} dx\Big)^{5/4}.\nonumber
\end{eqnarray}

On the other hand, by the Calder\'on-Zygmund estimate and a Sobolev interpolation inequality (see, e.g., (1.1) of \cite{LS})   we find
\begin{eqnarray}\label{ptil1}
\lefteqn{\int_{B_\rho(x_0)}|\tilde{p}_{x_0,\rho}|^{3/2} dx\leq C \int_{B_\rho(x_0)}| u-[u]_{x_0,\rho}|^{3} dx}\\
&\leq& C \Big(\int_{B_\rho(x_0)}| \nabla u|^{2} dx\Big)^{3/4} \Big(\int_{B_\rho(x_0)}|u-[u]_{x_0,\rho}|^{2} dx\Big)^{3/4}\nonumber\\
&\leq& C \Big(\int_{B_\rho(x_0)}| \nabla u|^{2} dx\Big)^{3/4} \Big(\int_{B_\rho(x_0)}|u|^{2} dx\Big)^{3/4},\nonumber
\end{eqnarray}
where we used  the bound 
$$\int_{B_r(x_0)}|u-[u]_{x_0,r}|^{2} dx\leq \int_{B_r(x_0)}|u|^{2} dx$$ 
in the last inequality.

Combining \eqref{pptil}, \eqref{ptil1} and using $r/\rho\leq 1/4$ we have 
\begin{eqnarray*}
\int_{B_r(x_0)} |p(x,t)|^{3/2} dx&\leq& C \Big(\int_{B_\rho(x_0)}| \nabla u|^{2} dx\Big)^{3/4} \Big(\int_{B_\rho(x_0)}|u|^{2} dx\Big)^{3/4}\\
&& +\,  C \frac{r^3}{\rho^{15/4}}\Big(\int_{B_{\rho}(x_0)} |p|^{6/5} dx\Big)^{5/4}.
\end{eqnarray*}

Integrating the last bound with respect to $dt/r^2$ over the interval $(t_0-r^2, t_0)$ and using H\"older's inequality we obtain 
\begin{eqnarray*}
D_1(z_0, r) &\leq& C \Big(\frac{\rho}{r}\Big)^{3/2} A(z_0,\rho)^{3/4} B(z_0,\rho)^{3/4}  + C \Big(\frac{r}{\rho}\Big)^{3/2} D(z_0,\rho)^{3/4}
\end{eqnarray*}
as desired.
\end{proof}

We are now ready to prove Proposition \ref{epsilon1}. 

\begin{proof}[{\bf Proof of Proposition \ref{epsilon1}}]
By Lemma \ref{boundC1} and Corollary \ref{ABcontrol2} we have  
\begin{eqnarray*}
C_1(z_0,1) &\leq& C   A(z_0,1)^{3/4} B(z_0,1)^{3/4} + C  A(z_0,1)^{3/2}\\
&\leq& C    [A(z_0,1) + B(z_0,1)]^{3/2}\\
&\leq& C    [C(z_0,2)^{1/2}+ C(z_0,2) + D(z_0,2)]^{3/2}.
\end{eqnarray*}

Thus using \eqref{smallnessCD} we find
\begin{eqnarray}\label{C1ep1}
C_1(z_0,1) &\leq&  C    (\epsilon_1^{1/2}+ \epsilon_1)^{3/2}.
\end{eqnarray}

On the other hand, using Lemma \ref{Dbar} with $r=1$ and $\rho=4$ there holds
\begin{equation*}
D_1(z_0, 1) \leq C  [A(z_0,4) + B(z_0,4)]^{3/2}+  C  D(z_0,4)^{3/4},
\end{equation*}
which by Corollary \ref{ABcontrol2} and \eqref{smallnessCD} yields
\begin{eqnarray*}
D_1(z_0, 1) &\leq& C  [C(z_0,8)^{1/2} + C(z_0,8)+ D(z_0,8)]^{3/2}+  C  D(z_0,8)^{3/4}\\
&\leq& C  \epsilon_1^{3/2}+  C  \epsilon_1^{3/4}.
\end{eqnarray*}

Now choosing $\epsilon_1$ sufficiently small  in \eqref{C1ep1} and the last bound, we can make 
\begin{equation*}
C_1(z_0,1) + D_1(z_0,1) \leq \epsilon_0,
\end{equation*}
and thus Lemma \ref{C1D1ep0} implies the desired regularity result.
\end{proof}

\section{Proof of Theorems \ref{localregularity} and \ref{nearMarc}}
This section is devoted to the proof of Theorems \ref{localregularity} and \ref{nearMarc}. We shall need the following lemma.

\begin{lemma}\label{weaktosuitable} Suppose that the pair of functions $(u,p)$ satisfies 
the Navier-Stokes equations in $Q_1(0,0)=B_1(0)\times(-1,0)$ in the sense of distributions and has the  properties \eqref{u-reg}, \eqref{p-assum}, and \eqref{serrinlorentz} for some  $q\in (3,\infty]$.
Then $(u,p)$ forms a suitable solution to the Navier-Stokes equations in $Q_{5/6}$ with a generalized energy {\it equality},
$u\in L^4(Q)$, and $p\in L^2(Q_{5/6})$.  Moreover,
the inequality
\begin{equation}\label{foralltbound}
\norm{u(\cdot, t)}_{L^{3, q}(B_{3/4})}\leq  \norm{u}_{L^{\infty}(-(3/4)^2,0; L^{3, q}(B_{3/4}))}
\end{equation}
holds for \emph{all} $t\in [-(3/4)^2, 0]$, and the function
$$t\rightarrow \int_{B_{3/4}} u(x,t) w(x) dx$$
is continuous on $[-(3/4)^2, 0]$ for any $w\in L^{3/2, q/(q-1)}(B_{3/4})$.  Here it is understood as usual that $q/(q-1)=1$ in the case $q=\infty$.
\end{lemma}

\begin{proof}
By  Sobolev inequality we have $u\in L^2(-1,0; L^6(B_1)),$
which  using \eqref{serrinlorentz} and the interpolative inequality  
$$\norm{u(\cdot, t)}_{L^4(B_1)} \leq C \norm{u(\cdot,t)}_{L^{3,q}(B_1)}^{\frac{1}{2}} \norm{u(\cdot,t)}_{L^6(B_1)}^{\frac{1}{2}}$$
yields 
\begin{equation}\label{uinLfour}
u\in L^4(Q).
\end{equation}

Thus by H\"older's inequality the nonlinear term 
\begin{equation}\label{non-reg}
u\cdot\nabla u \in L^{4/3}(Q).
\end{equation}

As above,  we have a decomposition
$$p=\tilde{p}+h,$$
where $\tilde{p}=R_iR_j[(u_i u_j)\chi_{B_1}]$, and  $h$ is harmonic in  $B_{1}$. By Calder\'on-Zygmund estimate we have 
\begin{equation}\label{ptilbound}
\norm{\tilde{p}}_{L^{2}(-1, 0; L^{2}(B_1))}\leq C \norm{u}_{L^{4}(-1, 0; L^{4}(B_1))}^2=C \norm{u}^2_{L^4(Q)},
\end{equation}
and by harmonicity and assumption \eqref{p-assum} there holds
\begin{eqnarray}\label{hbound}
\lefteqn{\norm{h}_{L^{2}(-1, 0; L^{\infty}(B_{5/6}))}\leq C \norm{h}_{L^{2}(-1, 0; L^{1}(B_{1}))}}\\
&=& C \norm{p-\tilde{p}}_{L^{2}(-1, 0; L^{1}(B_{1}))}\nonumber\\
&\leq& C \Big[\norm{p}_{L^{2}(-1, 0; L^{1}(B_{1}))}+ \norm{u}^2_{L^4(Q)}\Big].\nonumber
\end{eqnarray}

Estimates \eqref{ptilbound}--\eqref{hbound} imply in particular that the pressure 
\begin{equation}\label{p-reg}
p\in L^2(Q_{5/6}).
\end{equation}

Using the inclusions  \eqref{u-reg},  \eqref{uinLfour}, \eqref{non-reg}, \eqref{p-reg}, and the local interior regularity of  non-stationary Stokes systems we eventually find 
$$\int_{Q_{3/4}} (|u|^4 + |\partial_t u|^{4/3} + |\nabla^2 u|^{4/3} +|\nabla p|^{4/3}) dxdt <+\infty.$$

It then follows that 
$$u\in C(-[(3/4)^2, 0]; L^{4/3}(B_{3/4}))$$
and thus the function
$$g_{\varphi}(t):=\int_{B_{3/4}} u(x,t) \varphi(x) dx$$
is continuous on $[-(3/4)^2, 0]$ for any $\varphi\in C_0^{\infty}(B_{3/4})$. This yields
$$\Big|\int_{B_{3/4}} u(x,t) \varphi(x) dx\Big| \leq C \norm{\varphi}_{L^{3/2, q/(q-1)}(B_{3/4})} \norm{u}_{L^{\infty}(-(3/4)^2,0; L^{3, q}(B_{3/4}))}$$
for  {\it any} $t\in [-(3/4)^2, 0]$ and any $\varphi\in C_0^{\infty}(B_{3/4})$. Thus by the density of $C_0^{\infty}(B_{3/4})$ in $L^{3/2, q/(q-1)}(B_{3/4})$
we see that 
$$\norm{u(\cdot, t)}_{L^{3, q}(B_{3/4})}\leq  C \norm{u}_{L^{\infty}(-(3/4)^2,0; L^{3, q}(B_{3/4}))}$$
for  any $t\in [-(3/4)^2, 0]$. Then it can be seen, again by density, that  the function $g_{\varphi}(t)$ above is actually continuous on $[-(3/4)^2, 0]$ for any 
$\varphi\in L^{3/2, q/(q-1)}(B_{3/4})$.

Finally,  using  \eqref{uinLfour} and a standard mollification in $\RR^{3+1}$ combined with a truncation in time of test functions,   we obtain the local generalized energy equality 
in $Q_{5/6}$.
\end{proof}

We now proceed with the proof of Theorem \ref{localregularity}.

\vspace{.15in}
\noindent {\bf  Proof of Theorem \ref{localregularity}.} Henceforth, let the hypothesis of  Theorem \ref{localregularity} be enforced. Notice that by Lemma \ref{weaktosuitable} $(u,p)$ forms a suitable weak solution to the Navier-Stokes equations in $Q_{5/6}(0,0)$.
As in \cite{ESS1}, the proof of Theorem \ref{localregularity} goes by a contradiction.
Suppose that $z_0=(x_0, t_0)\in \overline{Q}_{1/2}(0,0)$ is a singular point. By definition, this means that there exists 
no neighborhood $\mathcal{N}$ of $z_0$ such that $u$ has a H\"older continuous representative on $\mathcal{N}\cap B_1(0)\times(-1,0])$.
By Lemma 3.3 of \cite{SS}, there exist $c_0>0$ and a sequence of numbers $\epsilon_k\in (0, 1)$ 
such that $\epsilon_k\rightarrow 0$ as $k\rightarrow +\infty$ and 
\begin{equation}\label{singcond}
A(z_0, \epsilon_k) = \sup_{t_0-\epsilon_k^2\leq s\leq t_0}\frac{1}{\epsilon_k}\int_{B(x_0,\epsilon_k)}|u(x,s)|^2 dx \geq c_0
\end{equation}
for any $k\in \NN$. Moreover, by Lemma \ref{weaktosuitable} we have in particular
\begin{equation}\label{att0}
u(\cdot, t_0)\in L^{3,q}(B_{3/4}(0)).
\end{equation}

Recall that we can decompose
$$p=\tilde{p}+h,$$
where $h$ is harmonic in  $B_{1}$, and 
$\tilde{p}= R_iR_j[(u_i u_j)\chi_{B_1}]$.

For each $Q=\om\times (a,b)$, where $\om\Subset\RR^3$ and $-\infty<a<b\leq 0$,   we 
choose a large $k_0=k_0(Q) \geq 1$ so that for any $k\geq k_0$ there hold
the implications
$$x\in \om \Longrightarrow x_0+\epsilon_k x\in B_{2/3},$$ 
and
$$t\in(a,b) \Longrightarrow t_0+\epsilon_k^2 t\in (-(2/3)^2, 0),$$
where the sequence $\{\epsilon_k\}$ is as in \eqref{singcond}.

Given such a $Q=\om\times (a,b)$, let us set
$$u_k(x, t)=\epsilon_k u(x_0+\epsilon_k x, t_0 +\epsilon_k^2 t),\quad p_k(x, t)=\epsilon_k^2 p(x_0+\epsilon_k x, t_0 +\epsilon_k^2 t),$$
and
$${\tilde p}_{k}(x, t)=\epsilon_k^2 \, \tilde{p}(x_0+\epsilon_k x, t_0 +\epsilon_k^2 t),\quad h_{k}(x, t)=\epsilon_k^2\, h(x_0+\epsilon_k x, t_0 +\epsilon_k^2 t)$$
for any $(x,t)\in Q$ and $k\geq k_0(Q)$.

The following proposition provides a non-trivial {\it ancient solution} (see \cite{Sere2} for this notion) that is essential in the proof of Theorem 
\ref{localregularity}.

\begin{proposition}\label{limit-infty} {\rm (i)} There exist subsequence of $(u_k,p_k)$, still denoted by $(u_k,p_k)$, and a pair of functions 
\begin{equation}\label{uinfpinf}
(u_\infty, p_\infty)\in L^{\infty}(-\infty,0; L^{3,q}(\RR^3)\times L^{\infty}(-\infty,0; L^{3/2,q/2}(\RR^3),
\end{equation}
with  ${\rm div}\, u_\infty=0$ in $\RR^3\times (-\infty, 0)$, such that 
\begin{equation}\label{Ccons}
u_k \rightarrow u_\infty \quad {\rm in} \quad C([a,b]; L^s(\om)),
\end{equation}
\begin{equation*}
\tilde{p}_k \rightarrow  p_\infty \quad {\rm weakly^{*} ~ in} \quad L^{\infty}(a,b; L^{3/2,q/2}(\om), 
\end{equation*}
for any $s\in(1,3)$, and any  $\om\Subset 
\RR^3$, $-\infty<a<b\leq0$.

\noindent {\rm (ii)} Moreover, for any $Q=\om\times (a,b)$ with $\om\Subset\RR^3$,  $-\infty<a<b\leq 0$,
$$|u_\infty|^2, \nabla u_\infty\in L^2(Q),\quad \partial_t u_\infty, \nabla^2 u_\infty, \nabla p_\infty\in L^{4/3}(Q),$$
and
$(u_\infty, p_\infty)$ forms a suitable weak solution of the Navier-Stokes equations in any such $Q$.

\noindent {\rm (iii)} Additionally, $u_\infty$ satisfies the lower bound
\begin{equation}\label{contradic-goal}
\sup_{t\in[-1,0]} \int_{B_1(0)} |u_\infty(x,t)|^2 dx \geq c_0,
\end{equation}
where $c_0>0$ is the constant in \eqref{singcond}.
\end{proposition}

\begin{proof}
For each $Q=\om\times (a,b)$, where $\om\Subset\RR^3$, $-\infty<a<b\leq 0$,
and for every $t\in [a,b]$ we have
\begin{equation}\label{ukboundinfty}
\norm{u_k(\cdot, t)}_{L^{3, q}(\om)}\leq \norm{u(\cdot, t_0+\epsilon_k^2 t)}_{L^{3, q}(B_{3/4})}\leq
\norm{u}_{L^{\infty}(-1, 0; L^{3, q}(B_1))}.
\end{equation}

By Calder\'on-Zygmund estimate, for a.e. $t\in (a,b)$ there holds
\begin{eqnarray}\label{ptil-k-bound}
\norm{{\tilde p}_k(\cdot, t)}_{L^{3/2, q/2}(\om)}&\leq& \norm{{\tilde p}(\cdot, t_0+\epsilon_k^2 t)}_{L^{3/2, q/2}(B_{3/4})}\\
&\leq& \esssup_{t'\in(-(3/4)^2,0)}\norm{{\tilde p}(\cdot, t')}_{L^{3/2, q/2}(B_{3/4})} \nonumber\\
&\leq& C \norm{u}_{L^{\infty}(-1, 0; L^{3, q}(B_1))}^2. \nonumber
\end{eqnarray}

On the other hand, by harmonicity  we have 
\begin{eqnarray*}
\int_{a}^{b}\sup_{x\in\om} |h_k(x,t)|^{2} dt&\leq& \epsilon_k^2 \int_{-(3/4)^2}^{0}\sup_{x\in\om} |h(x_0+\epsilon_k x,s)|^{2} ds\nonumber\\
&\leq& \epsilon_k^2 \norm{h}^{2}_{L^{2}(-1, 0; L^{\infty}(B_{3/4}))} \nonumber\\
&\leq& C\, \epsilon_k^2 \norm{h}^{2}_{L^{2}(-1, 0; L^{1}(B_{5/6}))} \nonumber
\end{eqnarray*}
provided $k\geq k_0(Q)$. Thus again  by Calder\'on-Zygmund estimate we find
\begin{eqnarray}\label{h-k-bound}
\lefteqn{\int_{a}^{b}\sup_{x\in\om} |h_k(x,t)|^{2} dt}\\
&\leq& C\, \epsilon_k^2 \norm{p-\tilde{p}}^{2}_{L^{2}(-1, 0; L^{1}(B_{5/6}))} \nonumber\\
&\leq& C \,\epsilon_k^2  \Big[\norm{p}_{L^{2}(-1, 0; L^{1}(B_{1}))}^2+ \norm{u}_{L^{\infty}(-1, 0; L^{3, q}(B_1))}^4\Big].\nonumber
\end{eqnarray}

Using the last estimates for ${\tilde p}_{k}$ and $h_k$ and H\"older's inequality we have the following uniform bound for $p_k$:
\begin{eqnarray}\label{pkbound}
\lefteqn{\norm{p_k}_{L^{2}(a, b; L^{6/5}(\om))}\leq \norm{{\tilde p}_k}_{L^{2}(a, b; L^{6/5}(\om))} + \norm{h_k}_{L^{2}(a, b; L^{6/5}(\om))}}\\
&\leq& C(Q)\Big[ \norm{p}_{L^{2}(-1, 0; L^{1}(B_{1}))}+ \norm{u}_{L^{\infty}(-1, 0; L^{3, q}(B_1))}^2 \Big] \nonumber
\end{eqnarray}
for any  $k\geq k_0(Q)$. Here the constant $C(Q)$ is independent of such $k$.

With regard to $u_k$,  with $k\geq k_0(Q)$, we have 
\begin{eqnarray}\label{ukbound}
\norm{u_k}_{L^{4}(a, b; L^{12/5}(\om))} &\leq& C(Q) \norm{u_k}_{L^{\infty}(a, b; L^{3, q}(\om))}\\
&\leq& C(Q)  \norm{u}_{L^{\infty}(-1, 0; L^{3, q}(B_1))}. \nonumber
\end{eqnarray}

For each $\varphi\in C_0^{\infty}(\RR^3\times \RR)$ that vanishes in a neighborhood of the parabolic boundary 
$\partial'Q=\om\times\{t=a\} \cup \partial\om\times [a, b]$, we  define
$$\varphi_k(x,t)=\epsilon_k^{-1}\varphi(\epsilon_k^{-1}(x-x_0), \epsilon_k^{-2}(t-t_0)).$$ 
Then with $k\geq k_0(Q)$ we see that $\varphi_k$  vanishes in a neighborhood of the parabolic boundary of $Q_{3/4}(0,0)$.
Using $\varphi_k$ as a test function in the generalized energy equality for $(u, p)$ at
$t=t_0+ \epsilon_k^2 \tau$ with a.e. $\tau\in (a, b)$ we find
\begin{eqnarray*}
\lefteqn{\int_{B_{3/4}}|u(x, t)|^2 \varphi_k(x,t) dx + 2\int_{-(3/4)^2}^t\int_{B_{3/4}} |\nabla u|^2 \varphi_k(x, s) dxds} \\
&=& \int_{-(3/4)^2}^t\int_{B_{3/4}} |u|^2 (\partial_t\phi_k +\Delta \phi_k) dx ds  \\
&& \quad +\, \int_{-(3/4)^2}^t\int_{B_{3/4}}(|u|^2 + 2p)u\cdot \nabla \varphi_k dx ds.
\end{eqnarray*}
Hence by making a change of variables we obtain

\begin{eqnarray*}
\lefteqn{ \int_{\om}|u_k(y, \tau)|^2 \varphi(y,\tau) dy + 2\int_{a}^{\tau}\int_{\om} |\nabla u_k|^2 \varphi(y, s') dyds'} \\
&=& \int_{a}^{\tau}\int_{\om} |u_k|^2 (\phi_t +\Delta \phi) dy ds' + \int_{a}^{\tau}\int_{\om}(|u_k|^2 + 2p_k)u_k\cdot \nabla \varphi dy ds'
\end{eqnarray*}
for a.e. $\tau \in (a,b)$.

Thus each $u_k$ is a suitable solution in $Q$ for any $Q=\om\times (a,b)$, with $\om\Subset\RR^3$ and $-\infty<a<b\leq 0$, and any $k\geq k_0(Q)$.
Then, given such a $Q$, it follows from Corollary \ref{ABcontrol2}  and inequalities \eqref{pkbound}--\eqref{ukbound}  (applied to an appropriate enlargement of $Q$) that 
\begin{equation}\label{ABkbound}
\norm{u_k}_{L^{\infty}(a, b; L^{2}(\om))}  + \norm{\nabla u_k}_{L^{2}(a, b; L^{2}(\om))}\leq C(Q)
\end{equation}
for all sufficiently large $k$ depending only on $Q$.

Using \eqref{ABkbound} and   Sobolev  inequality we have
\begin{equation*}
\norm{u_k}_{L^{2}(a, b; L^{6}(\om))}  \leq C(Q),
\end{equation*}
which by \eqref{ukboundinfty}, interpolation, and H\"older's inequality  gives
\begin{equation}\label{uknonlinearterm}
\norm{u_k}_{L^4(\om\times (a,b))} + \norm{u_k\cdot\nabla u_k}_{L^{4/3}(\om\times (a,b))}\leq C.
\end{equation}

From the   bounds \eqref{ptil-k-bound} and \eqref{h-k-bound}
for $\tilde{p}_k$ and $h_k$  we also have 
\begin{equation}\label{pk3halfs}
\norm{p_k}_{L^{s}(\om\times(a,b))}  \leq C(Q, s) \norm{p_k}_{L^{2}(a,b; L^{3/2, q/2}(\om))} \leq C
\end{equation}
for any $s\in (0, 3/2)$.

Using \eqref{ABkbound}--\eqref{pk3halfs}, it follows from the local interior regularity  of solutions to non-stationary Stokes equations 
we find 
\begin{equation}\label{partterm}
\norm{\partial_t u_k}_{L^{4/3}(\om\times (a,b))} + \norm{\nabla^2 u_k}_{L^{4/3}(\om\times (a,b))} + \norm{ \nabla p_k}_{L^{4/3}(\om\times (a,b))}\leq C
\end{equation}
for all sufficiently large $k$ depending only on $Q$.

At this point, using \eqref{ukboundinfty}--\eqref{ptil-k-bound} and a diagonal process we may assume that 
\begin{equation*}
u_k \rightarrow u_\infty \quad {\rm weakly^{*} ~ in} \quad L^{\infty}(a,b; L^{3,q}(\om) 
\end{equation*}
\begin{equation*}
\tilde{p}_k \rightarrow  p_\infty \quad {\rm weakly^{*} ~ in} \quad L^{\infty}(a,b; L^{3/2,q/2}(\om), 
\end{equation*}
for a pair of functions $(u_\infty, p_\infty)$ satisfying \eqref{uinfpinf},
with  ${\rm div}\, u_\infty=0$ in $\RR^3\times (-\infty, 0)$.

Estimates  \eqref{ABkbound} and \eqref{partterm} now yield 
\begin{equation}\label{con4thirds}
u_k \rightarrow u_\infty \quad {\rm in} \quad C([a,b]; L^{4/3}(\om)).
\end{equation}

For any $s\in (1, 3)$, the uniform bound \eqref{ukboundinfty}, and the interpolation inequality
\begin{eqnarray*}
\lefteqn{\norm{u_k(\cdot, t) -u_k(\cdot, t')}_{L^s(\om)}}\\
&\leq& C(s) \norm{u_k(\cdot, t) -u_k(\cdot, t')}_{L^{4/3}(\om)}^{\frac{12}{5}\left(\frac{1}{s}-\frac{1}{3}\right)} \norm{u_k(\cdot, t) -u_k(\cdot, t')}_{L^{3,q}(\om)}^{\frac{12}{5}\left(\frac{3}{4}-\frac{1}{s}\right)}
\end{eqnarray*}
imply that each $u_k\in C([a,b]; L^s(\om)$. Thus by using \eqref{con4thirds} and interpolating  we obtain
\eqref{Ccons} for any $s\in (1, 3)$. This completes the proof of  ${\rm (i)}$.

On the other hand, by \eqref{h-k-bound} we have 
\begin{equation*}
h_k \rightarrow 0 \quad {\rm strongly ~ in} \quad L^{2}(a,b; L^{\infty}(\om), 
\end{equation*}
for any $-\infty<a<b\leq 0$ and $\om\Subset\RR^3$, and thus in the limit $(u_\infty, p_\infty)$ satisfies the Navier-Stokes equations in the sense of 
distributions in $\om\times (a,b)$.  Now ${\rm (ii)}$ follows from ${\rm (i)}$, \eqref{ABkbound} and \eqref{partterm} via an argument as in the proof of Lemma \ref{weaktosuitable}.

Finally, note that by \eqref{singcond} and a change of variables we have 
$$\sup_{-1\leq t\leq 0}\int_{B(0,1)}|u_k(x,t)|^2 dx=\sup_{t_0-\epsilon_k^2\leq s\leq t_0}\frac{1}{\epsilon_k}\int_{B(x_0,\epsilon_k)}|u(y,s)|^2 dy \geq c_0.$$
Thus using \eqref{Ccons} with $s=2$ we obtain \eqref{contradic-goal}, which proves ${\rm (iii)}$.
\end{proof}


We now continue with the proof of Theorem \ref{localregularity}. By ${\rm (i)}$ of Proposition \ref{limit-infty}, we have 
$$\int_{-M}^{0} (\norm{u_\infty(\cdot, t)}_{L^{3,q}(\RR^3)}^{4} + \norm{p_\infty(\cdot, t)}_{L^{3/2,q/2}(\RR^3)}^{2})dt<+\infty$$
for any real number $M>0$. Note that for a.e. $t$,
$$\norm{u_\infty(\cdot, t)}_{L^{3,q}(\RR^3\setminus \overline{B}_R(0))}^{4} + \norm{p_\infty(\cdot, t)}_{L^{3/2,q/2}(\RR^3\setminus \overline{B}_R(0))}^{2}
\rightarrow 0$$
as $R\rightarrow +\infty$. We thus have 
$$\int_{-M}^{0} (\norm{u_\infty(\cdot, t)}_{L^{3,q}(\RR^3\setminus \overline{B}_R(0))}^{4} + \norm{p_\infty(\cdot, t)}_{L^{3/2,q/2}(\RR^3\setminus \overline{B}_R(0))}^{2})dt\rightarrow 0$$
as $R\rightarrow +\infty$. This yields that for any $M>200$ there exists $N=N(\epsilon_1, M)>10$ such that 
$$\int_{-M}^{0} (\norm{u_\infty(\cdot, t)}_{L^{3,q}(\RR^3\setminus \overline{B}_N(0))}^{4} + \norm{p_\infty(\cdot, t)}_{L^{3/2,q/2}(\RR^3\setminus \overline{B}_N(0))}^{2})dt\leq \epsilon_2,$$
where 
\begin{equation}\label{epsilon2}
\epsilon_2 =8^2|B_1(0)|^{-1/3}\epsilon_1
\end{equation}
with $\epsilon_1$ being as found in Proposition \ref{epsilon1}.

We now fix such numbers $M$ and $N$ and consider any $z_1$ that 
$$z_1=(x_1,t_1)\in (\RR^3\setminus \overline{B}_{2N}(0))\times (-M/2, 0].$$
Then there holds 
$$Q_8(z_1)=B_8(x_1)\times (t_1-8^2, t_1)\subset (\RR^3\setminus \overline{B}_{N}(0))\times (-M, 0],$$
and hence 
\begin{equation}\label{up3}
\int_{t_1-8^2}^{t_1} (\norm{u_\infty(\cdot, t)}_{L^{3,q}(B_8(x_1))}^{4} + \norm{p_\infty(\cdot, t)}_{L^{3/2,q/2}(B_8(x_1))}^{2})dt\leq \epsilon_2.
\end{equation}

Since 
\begin{eqnarray*}
\lefteqn{\norm{u_\infty(\cdot, t)}_{L^{12/5}(B_8(x_1))}^{4} + \norm{p_\infty(\cdot, t)}_{L^{6/5}(B_8(x_1))}^{2}}\\
&\leq& |B_8(x_1)|^{1/3} \left\{\norm{u_\infty(\cdot, t)}_{L^{3,q}(B_8(x_1))}^{4}+ \norm{p_\infty(\cdot, t)}_{L^{3/2,q/2}(B_8(x_1))}^{2}\right\}
\end{eqnarray*}
we see from \eqref{epsilon2}--\eqref{up3} that 
\begin{equation}\label{CDz1}
C(u_\infty,z_1,8) + D(p_\infty,z_1,8) \leq \epsilon_1.
\end{equation}

The smallness property \eqref{CDz1} and Proposition \ref{epsilon1} now yield that $\nabla^k u_\infty$, $k=0,1,2, \dots$, is H\"older continuous 
on $(\RR^3\setminus \overline{B}_{2N}(0))\times (-M/2, 0]$, and 
\begin{equation}\label{reg-bound}
\max_{z\in \overline{Q}_{1/2}(z_1)} |\nabla^{k} u_\infty(z)| \leq C_k.
\end{equation}

Let $\om_\infty={\rm curl}\, u_\infty$ be the vorticity of $u_\infty$. Then $\om_\infty$ satisfies the equation 
$$\partial_t \om_\infty- \Delta \om_\infty + (u_{\infty} \cdot \nabla) \om_\infty - (\om_{\infty}\cdot \nabla) u_\infty=0$$
on the set $(\RR^3\setminus \overline{B}_{4N}(0))\times (-M/4, 0]$, which by \eqref{reg-bound} gives
\begin{equation}\label{vorticbound1}
|\partial_t \om_\infty- \Delta \om_\infty|\leq C(|\om_\infty| + |\nabla \om_\infty|) 
\end{equation}
with  
\begin{equation}\label{vorticbound2}
|\om_\infty| \leq C<+\infty
\end{equation}
on the set  $(\RR^3\setminus \overline{B}_{4N}(0))\times (-M/4, 0]$, for a universal  constant $C>0$.

We now claim that 
\begin{equation}\label{zerooutside}
\om_\infty=0 \quad {\rm on~} (\RR^3\setminus \overline{B}_{4N}(0))\times (-M/4, 0].
\end{equation}

To see this, by applying the backward uniqueness theorem (see \cite[Theorem 5.1]{ESS1} and \cite{ESS2}) and the bounds 
\eqref{vorticbound1}--\eqref{vorticbound2}, it is enough to show that 
\begin{equation}\label{vorzeroat0}
\om_\infty(y,0)=0 \quad {\rm for ~ all~} y\in \RR^3\setminus \overline{B}_{4N}(0)).
\end{equation}

Note that   for any $y\in \RR^3$ we have 
\begin{eqnarray*}
\lefteqn{\int_{B_1(y)} |u_\infty(x, 0)| dx}\\
&\leq& \int_{B_1(y)} |u_\infty(x, 0)-u_k(x,0)| dx + \int_{B_1(y)} |u_k(x, 0)| dx\\
&\leq& \int_{B_1(y)} |u_\infty(x, 0)-u_k(x,0)| dx + |B_1(0)|^{\frac{2}{3}}\norm{u_k(\cdot,0)}_{L^{3,q}(B_1(y))}\\
&\leq& \norm{u_\infty-u_k}_{C([-M/4, 0]; L^1(B_1(y)))} + |B_1(0)|^{\frac{2}{3}}\norm{u(\cdot,t_0)}_{L^{3,q}(B_{\epsilon_k}(x_0+\epsilon_k y))}.
\end{eqnarray*}

Thus sending $k\rightarrow +\infty$ we see that 
$$\int_{B_1(y)} |u_\infty(x, 0)| dx=0 $$
for all $y\in \RR^3$, which yields \eqref{vorzeroat0} as desired. Here we have used ${\rm (i)}$ of Proposition \ref{limit-infty} and \eqref{att0}.

At this point using \eqref{zerooutside} combined with the argument on pages 227--229 of \cite{ESS1}, which ultimately employs 
the theory of unique continuation for parabolic inequalities, we see that in fact 
$$\om_\infty(\cdot,t)=0\quad {\rm in~ the~ whole~} \RR^3$$
 for a.e. $t\in (-M/4, 0)$. Thus $u_\infty(\cdot,t)$ is globally  harmonic and by a Liouville theorem it follows that  $u_\infty(\cdot,t)=0$ for a.e. 
$t\in (-M/4, 0)$. This leads to a contradiction to the lower bound \eqref{contradic-goal} and hence completes the proof of  Theorem \ref{localregularity}.

We next prove Theorem \ref{nearMarc}.

\begin{proof}[{\bf Proof of Theorem \ref{nearMarc}}] Arguing as in the proof of Lemma 
\ref{weaktosuitable} we see that $(u,p)$ forms a suitable solution to the Navier-Stokes equations in $Q_{5/6}$.

Suppose that $(x_0,t_0)\in \overline{Q}_{1/2}(0,0)$ is a singular point.
Then  we must have that $x_0=0$. 
By Lemma 3.3 of \cite{SS}, there exist $c_0>0$ and a sequence of numbers $\epsilon_k\in (0, 1/8)$ 
such that $\epsilon_k\rightarrow 0$ as $k\rightarrow +\infty$ and 
\begin{equation*}
\sup_{t_0-\epsilon_k^2\leq s\leq t_0}\frac{1}{\epsilon_k}\int_{B(0,\epsilon_k)}|u(x,s)|^2 dx \geq c_0.
\end{equation*}

Using a change of variables and the condition \eqref{LinftyX}, we then have
\begin{eqnarray*}
0<c_0 &\leq& \sup_{-1\leq t\leq 0} \int_{B(0,1)}|\epsilon_k u(\epsilon_k y, t_0+\epsilon_k^2 t)|^2 dy\\
&\leq& \norm{f}^2_{L^\infty((-1,0))} \int_{B_1(0)} |y|^{-2} g(\epsilon_k y)^2 dy.
\end{eqnarray*}

By the property of $g$, this is impossible to hold for all $k\in \NN$, and thus  the proof of  Theorem \ref{nearMarc} is complete.
\end{proof}

\section{Proof of Theorems \ref{global-ESS-lorentz}}
We shall prove  Theorems \ref{global-ESS-lorentz}  in this section.  First observe that under the hypothesis of Theorem \ref{global-ESS-lorentz}, we have 
\begin{eqnarray*}
\norm{u(\cdot, t)}_{L^4(\RR^3)} &\leq& C \norm{u(\cdot,t)}_{L^{3,q}(\RR^3)}^{\frac{1}{2}} \norm{u(\cdot,t)}_{L^6(\RR^3)}^{\frac{1}{2}}\\
&\leq& C \norm{u(\cdot,t)}_{L^{3,q}(\RR^3)}^{\frac{1}{2}} \norm{\nabla u(\cdot,t)}_{L^2(\RR^3)}^{\frac{1}{2}}.
\end{eqnarray*}

Thus,
\begin{equation}\label{L4}
u\in L^4(Q_T) \quad {\rm and} \quad u\cdot\nabla u\in L^{4/3}(Q_T),
\end{equation}
where the latter follows from H\"older's inequality. Using these inclusions, the coercive estimates (see \cite{GS, MS, Sol})  and the uniqueness theorem (see \cite{Lad1}) for the Stokes problem, we can introduce the  {\it associated pressure} $p$ such that  
$$\partial_t u, \nabla^2 u, \nabla p\in L^{4/3}(\RR^3\times (\delta, T))$$
for any $\delta \in (0, T)$. Moreover, it follows from the pressure equation and the global condition \eqref{PSLcrit-lorentz} that 
\begin{equation}\label{presurelorentz}
p\in L^\infty(0, T; L^{3/2, q/2}(\RR^3)).
\end{equation}

Arguing as in the proof of Lemma \ref{weaktosuitable} we see that $(u,p)$ forms a suitable weak solution in any bounded  cylinder of $Q_T$. By 
\eqref{PSLcrit-lorentz} and \eqref{presurelorentz}, we have  
\begin{equation}\label{T-for-decay}
\int_{0}^{T} (\norm{u(\cdot, t)}_{L^{3,q}(\RR^3)}^{4} + \norm{p(\cdot, t)}_{L^{3/2,q/2}(\RR^3)}^{2})dt\leq C(T)<+\infty.
\end{equation}

We next fix a $\delta\in (0, T)$ and set $r_0=\sqrt{\delta/128}$. Then using \eqref{T-for-decay} we find a large number $R=R(T,\delta)>0$ such that  
\begin{equation*}
C(u, z_0, 8r_0) +  D(p, z_0, 8r_0)\leq \epsilon_1
\end{equation*}
for any $z_0=(x_0,t_0)\in \RR^3\setminus B_R(0)\times [\delta, T]$. Thus by scaling and Proposition \ref{epsilon1}, there holds 
$$\sup_{\RR^3\setminus B_R(0)\times[\delta, T]} |u|\leq C(\delta).$$ 

On the other hand, for any $z_0=(x_0,t_0)\in  \overline{B}_R(0)\times [\delta, T]$ and with $r_0$ as above, $(u,p)$ is a suitable solution in
$Q_{r_0}(z_0)$. Thus by scaling, Theorem \ref{localregularity}, and the compactness of $\overline{B}_R(0)\times [\delta, T]$, we have 
$$\sup_{\overline{B}_R(0)\times[\delta, T]} |u|\leq C(\delta).$$

Combining the last two bounds we obtain 
\begin{equation*}
\sup_{\RR^3\times[\delta, T]} |u|\leq C(\delta)<+\infty
\end{equation*}
which holds for any $\delta\in(0,T)$. Thus $u$ is smooth on $\RR^3\times (0,T]$, and  using $u\in L^4(Q_T)$ (see \eqref{L4}) and interpolation we see that $u\in L^5(\RR^3\times(\delta,T))$
for any $\delta\in (0, T)$.

On the other hand, if $a\in \dot{J}\cap L^3$  then by local strong solvability  and weak-strong uniqueness $u\in L^5(\RR^3\times(0,\delta_0))$ for some $\delta_0>0$ (see, e.g., \cite[Theorem 7.4]{ESS1} and \cite{Kat}).
Thus we conclude that $u\in L^5(\RR^3\times(0,T))$ and hence by Theorem \ref{PSL} it is unique in $Q_T$.

\begin{center} {\sc Acknowledgements}
\end{center}

The author wishes to acknowledge the  support from the Hausdorff Research Institute for Mathematics  (Bonn, Germany),  where  this 
work has been finalized. Part of this work was also done during a visit  to the Institut Mittag-Leffler (Djursholm,
Sweden).


\begin{thebibliography}{xx}
\bibitem{CKN} L. Caffarelli, R.-V. Kohn, and L. Nirenberg, {\it Partial regularity of suitable weak solutions of
the Navier-Stokes equations}, Comm. Pure Appl. Math. {\bf 35} (1982), 771--831.

\bibitem{CP} J.-Y. Chemin and F.  Planchon, {\it Self-improving bounds for the Navier-Stokes equations},
Bull. Soc. Math. France {\bf 140} (2012) 583–-597.

\bibitem{CSTY1} C.-C. Chen, R. M. Strain, T.-P. Tsai, and H.-T. Yau, {\it Lower bounds on the blow-up rate of the axisymmetric Navier-Stokes equations
II}, Comm. Partial Differential Equations, {\bf 34} (2009), 203--232.

\bibitem{CSTY2} C.-C. Chen, R. M. Strain, H.-T. Yau, and T.-P. Tsai, {\it Lower bound on the blow-up rate of the axisymmetric Navier-Stokes equations},
Int. Math. Res. Not. IMRN,  2008, no. 9, Art. ID rnn016, 31 pp.

\bibitem{ESS1} L. Escauriaza,  G. Seregin, and V. {\v S}ver\'ak,  {\it $L_{3, \infty}$-solutions of Navier-Stokes equations and backward uniqueness}, (Russian) Uspekhi Mat. Nauk {\bf 58} (2003),  3--44; 
translation in Russian Math. Surveys {\bf 58} (2003),  211--250.

\bibitem{ESS2} L. Escauriaza, G. Seregin, and V. {\v S}ver\'ak, {\it Backward uniqueness for parabolic equations}, 
Arch. Ration. Mech. Anal. {\bf 169} (2003),  147--157.

\bibitem{GKP} I. Gallagher, G.  Koch, and F. Planchon,  {\it A profile decomposition approach to the $L^\infty_t(L^3_x)$ Navier-Stokes regularity 
criterion}, Math. Ann. {\bf 355} (2013),  1527--1559.

\bibitem{GKP2} I. Gallagher, G.  Koch, and F. Planchon,  {\it Blow-up of critical Besov norms at a potential Navier-Stokes singularity},
Preprint 2014.  arXiv:1407.4156.  


\bibitem{GS} Y. Giga and H. Sohr, {\it Abstract $L^p$-estimates for the Cauchy problem with applications to
the Navier-Stokes equations in exterior domains}, J. Funct. Anal. {\bf 102} (1991), 72--94.

\bibitem{Giu} E. Giusti, {\it Direct methods in the calculus of variations}, World Scientific Publishing Co., Inc., River Edge, NJ, 2003.


\bibitem{Hop} E. Hopf, {\it \"Uber die Anfangswertaufgabe f\"ur die hydrodynamischen Grundgleichungen},
Math. Nachr. {\bf 4} (1951), 213--231.

 
\bibitem{Kat} T. Kato, {\it Strong $L^p$-solutions of the Navier-Stokes equations in $R^m$ with applications to
weak solutions}, Math. Z. {\bf 187} (1984), 471--480.

\bibitem{KeKo} C.  Kenig and G. Koch, {\it An alternative approach to the Navier-Stokes 
equations in critical spaces}, Ann. Inst. H. Poincar\'e Anal. Non Lin\'eaire {\bf 28} (2011),  159--187.


\bibitem{KK} H. Kim and H. Kozono, {\it Interior regularity criteria in weak spaces for the Navier-Stokes equations},
 Manuscripta Math. {\bf 115} (2004), 85--100.

\bibitem{KS} H. Kozono and H. Sohr, {\it Remark on uniqueness of weak solutions to the Navier-Stokes
equations}, Analysis {\bf 16} (1996), 255--271. 

\bibitem{Lad1} O. A. Ladyzhenskaya, {\it Mathematical problems of the dynamics of viscous incompressible
fluids}, Gos. Izdat. Fiz.-Mat. Lit., Moscow 1961; 2nd rev. aug. ed. of English transl.,
{\it The mathematical theory of viscous incompressible flow}, Gordon and Breach, 
New York-London 1969.

\bibitem{Lad} O.  Ladyzhenskaya, {\it On uniqueness and smoothness of generalized solutions to the Navier-Stokes 
equations}, Zap. Nauchn. Sem. Leningrad. Otdel. Mat. Inst. Steklov. (LOMI) {\bf 5}
(1967), 169--185; English transl., Sem. Math. V.A. Steklov Math. Inst. Leningrad {\bf 5} (1969),
60--66.

\bibitem{LS} O.  Ladyzhenskaya and G.  Seregin, {\it On partial regularity of suitable weak solutions to
the three-dimensional Navier–Stokes equations}, J. Math. Fluid Mech. {\bf 1} (1999), 356--387.

\bibitem{Ler} J. Leray, {\it Sur le mouvement d'un liquide visqueux emplissant l'espace}, Acta Math. {\bf 63}
(1934), 193--248.

\bibitem{Lin} F.-H. Lin, {\it A new proof of the Caffarelli-Kohn-Nirenberg theorem}, Comm. Pure Appl.
Math. {\bf 51} (1998), 241--257.

\bibitem{LT} Y. Luo and T.-P. Tsai,  {\it Regularity criteria in weak $L^3$ for $3D$ incompressible
Navier-Stokes equations}, Preprint 2014. arXiv:1310.8307v5.

\bibitem{MS} P. Maremonti and V. A. Solonnikov, {\it On estimates for the solutions of the nonstationary
Stokes problem in anisotropic Sobolev spaces with a mixed norm}, Zap. Nauchn. Sem.
S.-Peterburg. Otdel. Mat. Inst. Steklov. (POMI) {\bf 222} (1995), 124--150; English transl.,
J. Math. Sci. (New York) {\bf 87} (1997), 3859--3877.

\bibitem{NRS} J. Ne{\v c}as, M. R{\r u}{\v z}i{\v c}ka, and V. {\v S}ver\'ak {\it On Leray's self-similar solutions of the Navier-Stokes
equations}, Acta Math. {\bf 176} (1996), 283--294.

\bibitem{Pro} G. Prodi, {\it Un teorema di unicit\`a per le equazioni di Navier-Stokes}, Ann. Mat. Pura Appl.
(4) {\bf 48} (1959), 173--182.

\bibitem{Sche1} V. Scheffer, {\it Partial regularity of solutions to the Navier-Stokes equations}, Pacific J. Math.
{\bf 66} (1976), 535--552.

\bibitem{Sche2} V. Scheffer, {\it Hausdorff measure and the Navier-Stokes equations}, Comm. Math. Phys. {\bf 55}
(1977), 97--112.

\bibitem{Sche3} V. Scheffer, {\it The Navier-Stokes equations in a bounded domain}, Comm. Math. Phys. {\bf 73}
(1980), 1--42.

\bibitem{Sche4} V. Scheffer, {\it Boundary regularity for the Navier-Stokes equations in a half-space}, Comm.
Math. Phys. {\bf 85} (1982), 275--299.


\bibitem{SS} G. Seregin and V. {\v S}ver\'ak, {\it Navier-Stokes equations with lower bounds on the pressure}, Arch. Ration. Mech. Anal. {\bf 163} (2002),  65--86. 

\bibitem{SS2} G. Seregin and V. {\v S}ver\'ak, {\it On type I singularities of the local axi-symmetric
solutions of the Navier-Stokes equations}, Comm. Partial Differential Equations,
{\bf 34} (2009), 171--201.

\bibitem{Sere} G. Seregin, {\it A certain necessary condition of potential blow up for Navier-Stokes equations},
 Commun. Math. Phys. {\bf 312} (2012), 833--845. 

\bibitem{Sere2} G. Seregin, {\it Selected topics of local regularity theory for the Navier-Stokes equations}. Topics in mathematical 
fluid mechanics, 239--313, Lecture Notes in Math. {\bf 2073}, Springer, Heidelberg, 2013.

\bibitem{Ser1} J. Serrin, {\it On the interior regularity of weak solutions of the Navier–Stokes equations},
Arch. Ration. Mech. Anal. {\bf 9} (1962), 187--195.

\bibitem{Ser2} J. Serrin, {\it The initial value problem for the Navier–Stokes equations}, Nonlinear Problems
(R. Langer, ed.), Univ. of Wisconsin Press, Madison 1963, pp. 69--98

\bibitem{Soh} H. Sohr, {\it A regularity class for the Navier-Stokes equations in Lorentz spaces}, J.
Evol. Equ. {\bf 1} ( 2001), 441--467.

\bibitem{Sol} V. A. Solonnikov, {\it Estimates of the solutions of a nonstationary linearized system of
Navier-Stokes equations}, Trudy Mat. Inst. Steklov. {\bf 70} (1964), 213--317; English transl., Amer.
Math. Soc. Transl. (2) {\bf 75} (1968), 1--116.



\bibitem{Stru} M. Struwe, {\it On partial regularity results for the Navier-Stokes equations}, Comm. Pure
Appl. Math. {\bf 41} (1988), 437--458.

\bibitem{Tak} S. Takahashi, {\it On interior regularity criteria for weak solutions of the Navier--Stokes
equations}, Manuscripta Math. {\bf 69} (1990), 237--254.

\end{thebibliography}
\end{document}